\documentclass[12pt,a4paper,oneside]{amsart}
\usepackage{amsfonts, amsmath, amssymb, amsthm}
\usepackage[colorlinks=true,citecolor=blue]{hyperref}
\usepackage[margin=1.4in]{geometry}
\usepackage{txfonts}
\usepackage[T1]{fontenc}
\usepackage{bbm}
\usepackage{mathtools}
\usepackage{graphicx}
\usepackage{enumerate}
\usepackage{verbatim}
\usepackage{tikz}
\usepackage{subcaption}

\begingroup
    \makeatletter
    \@for\theoremstyle:=definition,remark,plain\do{%
        \expandafter\g@addto@macro\csname th@\theoremstyle\endcsname{%
            \addtolength\thm@preskip\parskip
            }%
        }
\endgroup

\newtheorem{theorem}{Theorem}
\newtheorem*{theorem*}{Theorem}
\newtheorem{lemma}[theorem]{Lemma}

\newtheorem{defs}[theorem]{Definition}
\newtheorem{corollary}[theorem]{Corollary}
\newtheorem{question}[theorem]{Question}
\newtheorem*{claim*}{Claim}
\newtheorem*{quest*}{Question}

\theoremstyle{definition}

\newtheorem*{remark*}{Remark}

\newtheorem{obs}{Observation}

\newcommand{\size}[1]{\left| #1 \right|}

\makeatletter
\providecommand*{\cupdot}{%
  \mathbin{%
    \mathpalette\@cupdot{}%
  }%
}
\newcommand*{\@cupdot}[2]{%
  \ooalign{%
    $\m@th#1\cup$\cr
    \hidewidth$\m@th#1\cdot$\hidewidth
  }%
}
\makeatother

\title{Transversal numbers of stacked spheres}
\author{Minho Cho$^1$}
\address[Minho Cho]{Extremal Combinatorics and Probability Group, IBS, Daejeon, South Korea}
\email{minhocho@ibs.re.kr}

\author{Jinha Kim$^2$}
\address[Jinha Kim]{Discrete Mathematics Group, IBS, Daejeon, South Korea}
\email{jinhakim@ibs.re.kr}

\thanks{$^1$ Minho Cho was supported by the Institute for Basic Science~(IBS-R029-C4).}
\thanks{$^2$ Corresponding author. Jinha Kim was supported by the Institute for Basic Science (IBS-R029-Y5).}
\thanks{Part of this work was conducted when the first author was a Ph.D student at KAIST and included in his dissertation~\cite{Cho2023}.}
\date{\today}

\begin{document}

\maketitle

\begin{abstract}
A stacked $d$-sphere $S$ is the boundary complex of a stacked $(d+1)$-ball, which is obtained by taking cone over a free $d$-face repeatedly from a $(d+1)$-simplex.
A stacked sphere $S$ is called linear if every cone is taken over a face added in the previous step.
In this paper, we study the transversal number of facets of stacked $d$-spheres, denoted by $\tau(S)$, which is the minimum number of vertices intersecting with all facets.
Briggs, Dobbins and Lee showed that the transversal ratio of a stacked $d$-sphere is bounded above by $\frac{2}{d+2}+o(1)$ and can be as large as $\frac{2}{d+3}$.
We improve the lower bound by constructing linear stacked $d$-spheres with transversal ratio $\frac{6}{3d+8}$ and general stacked $d$-spheres with transversal ratio $\frac{2d+3}{(d+2)^2}$.
Finally, we show that $\frac{6}{3d+8}$ is optimal for linear stacked $2$-spheres, that is, the transversal ratio is at most $\frac{3}{7} + o(1)$ for linear stacked $2$-spheres.
\end{abstract}

\section{Introduction}
A {\em hypergraph} on $V$ is a collection of subsets of $V$.
For a hypergraph $H$ on $V$, we say $V=V(H)$ is the {\em vertex set} of $H$ and each element $e \in H$ is called an {\em edge} of $H$.
We say a vertex set $T \subset V$ is a {\em transversal} of $H$ if $e \cap T \neq \emptyset$ for every edge $e \in H$ and we define the {\em transversal number} $\tau(H)$ of $H$ as the minimum size of a transversal of $H$.
In general, the transversal number of a hypergraph can be as large as the size of the vertex set up to a constant difference: for instance, if $H=\binom{V}{d}$, then $\tau(H) = |V|-d+1$.
When we have a given class of hypergraphs, it is natural to ask how large the ``transversal ratio'' $\tau(H) / |V(H)| \in (0,1)$ can be over all hypergraphs $H$ in the class.

Among various classes of hypergraphs, we are mainly interested in facet hypergraphs of polytopes and their transversal numbers.
(For more references related to transversals of hypergraphs from geometry, see~\cite{AKMM02, GPW}.)
Let $P \subset \mathbb{R}^d$ be a $d$-dimensional polytope with the vertex set $V=V(P)$.
A {\em face} of $P$ is given as $h \cap P$ for some hyperplane $h$ such that one of the two closed half-spaces of $h$ contains $P$.
A {\em facet} of $P$ is a $(d-1)$-dimensional face of $P$.
We define the {\em facet hypergraph} of $P$ as the hypergraph on $V$ whose edges are the vertex sets of the facets of $P$.

Let $\mathcal{H}_d$ be the class of all facet hypergraphs of $d$-dimensional polytopes.
Then how large can the transversal ratio $\tau(H)/|V(H)|$ be over all $H \in \mathcal{H}_d$?
It is known that $\tau(H) \le |V(H)|/2$ for all $H \in \mathcal{H}_3$, which can be easily verified using the $4$-color theorem. \cite{BDL21+}
However, if $d \ge 4$, then answering this question seems challenging.
It is even unknown whether the transversal ratio can be bounded away from $1$ for any $d \ge 4$.
\begin{question}[Question 1, \cite{BDL21+}]\label{q:general}
Let $d \ge 4$ be an integer.
How large can the transversal ratio ${\tau(H)}/{n}$ be for sufficiently large $n$ and $H \in \mathcal{H}_d$ on $n$ vertices?
\end{question}

To answer Question~\ref{q:general}, we can ask whether there is a constant $c=c(d)<1$ such that for sufficiently large $n$ and every $H \in \mathcal{H}_d$ on $n$ vertices, $\tau(H) \le cn$.
A quick reduction is that to find such a constant $c(d)$, it suffices to consider only \emph{simplicial polytopes} whose facets are all simplices. 
This is because for a given polytope $P$, by moving some vertices of $P$ slightly towards its interior, we get a simplicial polytope whose facet hypergraph is isomorphic to a refinement of that of $P$.

In \cite{BDL21+}, the authors studied the transversal ratio for classes of simplicial polytopes.
As a special class of simplicial polytopes, they also studied the class of ``stacked balls''. (Some literature, for instance \cite{MRS}, refers to these as stacked polytopes.)

A {\em simplicial complex} $K$ on $V$ is a collection of subsets of $V$ that is closed under taking subsets, and each element in $K$ is called a {\em face} of $K$ and a {\em facet} of $K$ is a maximal face of $K$.
A face $\sigma \in K$ is called {\em free} if there is an unique facet containing $\sigma$ in $K$.
The {\em dimension} of a face $\sigma \in K$ is $\dim(\sigma)=|\sigma|-1$ and the {\em dimension} of $K$ is the maximum dimension of its faces.
We say a simplicial complex $K$ on $V$ is a {\em simplex} if $K=2^V$ and a $d$-dimensional simplex is just called a {\em $d$-simplex}.
When $\sigma$ is a set with $d+1$ vertices, then we denote by $\overline{\sigma}$ the $d$-simplex on $\sigma$.
The {\em facet hypergraph} $H(K)$ of a simplicial complex $K$ on $V$ is a hypergraph on $V$ whose edges are the facets of $K$.
\[H(K):=\{\sigma \subset V: \sigma \text{ is a facet of }K\}.\]
The {\em transversal number} $\tau(K)$ of $K$ is the transversal number of its facet hypergraph, i.e. $\tau(K)=\tau(H(K))$.

A {\em stacked $d$-ball} is a $d$-dimensional simplicial complex that can be obtained from a $d$-simplex by repeatedly gluing another $d$-simplex onto its $(d-1)$-dimensional free faces.
More precisely, it can be described as follows.

\begin{defs}
A simplicial complex $B$ is called a stacked $d$-ball if there is a sequence of simplicial complexes $B_1,\dots,B_m=B$ and a sequence of $d$-simplices $(\sigma_1,\dots,\sigma_m)$ satisfying the following:
\begin{enumerate}[(i)]
    \item $B_1=\sigma_1$.
    \item For each $i \in [m-1]$, $B_{i+1}=B_i \cup \sigma_{i+1}$ and $B_i \cap \sigma_{i+1}=\eta_{i+1}$ is a simplex on a $(d-1)$-dimensional free face of $B_i$.
\end{enumerate}
We denote by $\Gamma(B)=(\sigma_1,\dots,\sigma_m)$ the sequence of $d$-simplices constructing $B$.
\end{defs}

Let $B$ be a stacked $(d+1)$-ball where $\Gamma(B)=(\sigma_1,\dots,\sigma_m)$ for some $(d+1)$-simplices $\sigma_1,\dots,\sigma_m$.
Then we have $|V(B)|=d+m+1$.
We can consider the $d$-dimensional simplicial complex $S$ on $V(B)$ where the facets of $S$ are the $d$-dimensional free faces of $B$, and we call $S$ is a {\em stacked $d$-sphere} and denote $S=\partial B$.
If $\Gamma(B)=(\sigma_1,\dots,\sigma_m)$ and $S=\partial B$, then we also write $S=\partial(\sigma_1,\dots,\sigma_m)$.
Note that a stacked $(d+1)$-ball can be considered as a $(d+1)$-dimensional polytope, and a stacked $d$-sphere is a special case of a {\em simplicial $d$-sphere}, which is a simplicial complex homeomorphic to the $d$-dimensional sphere.

Now, we define the {\em dual graph} $G(B)$ of a stacked $(d+1)$-ball $B$ where $\Gamma(B)=(\sigma_1,\dots,\sigma_m)$.
The dual graph $G(B)$ is a graph on the vertex set $[m]$ and two vertices $i,j \in [m]$ are adjacent if and only if $\sigma_i \cap \sigma_j$ is a $d$-simplex.
Equivalently, when we construct $B$, we add the edge $\{i,j\}$ for $1 \le i<j \le m$ to $G(B)$ if $\sigma_j$ is attached to $\sigma_i$, i.e. $\eta_j \subset \sigma_i$.
Note that by the definition of a stacked ball, $G(B)$ is a tree for every stacked ball $B$, that is, $G(B)$ is connected and does not contain a cycle.
Now, as a special case, we say a stacked $d$-sphere $S=\partial B$ is {\em linear} if $G(B)$ is a path.

\begin{figure}[htbp]
    \centering
    \includegraphics[scale=0.175]{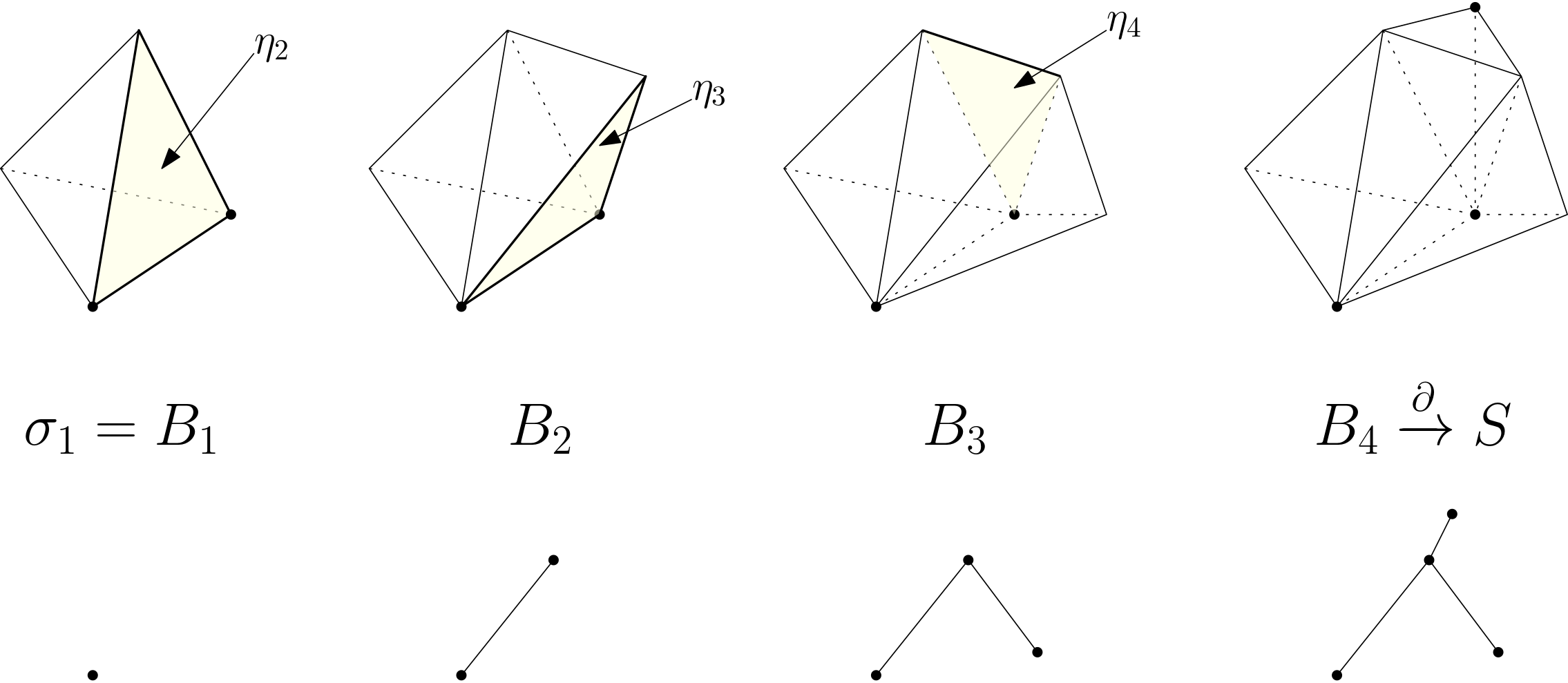}
    \caption{Construction of a stacked 2-sphere $S$ and its dual graph.}
    \label{fig_1}
\end{figure}

Our main concern is the following question.
\begin{question}[Question 3, \cite{BDL21+}]\label{que:ratio}
What are the following values for a fixed $d \ge 2$?
\begin{enumerate}
    \item $\displaystyle T(d):=\limsup_{n\to\infty}\max\left\{\frac{\tau(S)}{n}: S\text{ is a stacked }d\text{-sphere on }n\text{ vertices}\right\}$.
    \item $\displaystyle T^{\text{path}}(d):=\limsup_{n\to\infty}\max\left\{\frac{\tau(S)}{n}: S\text{ is a linear stacked }d\text{-sphere on }n\text{ vertices}\right\}$.
\end{enumerate}
\end{question}

Briggs et al.~\cite{BDL21+} showed that $T(d) \le \frac{2}{d+2}$ for every $d \ge 2$.
They also constructed an infinite family of linear stacked $d$-spheres that shows $T^{\text{path}}(d) \ge \frac{2}{d+3}$, and asked whether this is the optimal value.
We figured out that this is actually not true, by constructing an infinite family of linear stacked $d$-spheres showing $T^{\text{path}}(d) \ge \frac{6}{3d+8}$.
\begin{theorem}[Lower bound]\label{thm:lowerpath}
Let $d \ge 2$ and $k \ge 1$ be integers.
There exists a linear stacked $d$-sphere $S_d$ on $(3d+8)k$ vertices such that $\tau(S_d) \ge 6k$.
\end{theorem}

We also proved that the bound $\frac{6}{3d+8}$ is actually optimal for linear stacked $2$-spheres, that is, $T^{\text{path}}(2)= \frac{3}{7}$.
\begin{theorem}[Upper bound]\label{thm:upper}
Let $n \ge 4$ be an integer and let $S$ be a linear stacked $2$-sphere on $n$ vertices. Then there is a transversal of $S$ with at most $\lceil\frac{3n}{7}\rceil$ vertices.
\end{theorem}

If we consider general stacked $d$-spheres, then Theorem~\ref{thm:lowerpath} also gives a lower bound $\frac{6}{3d+8}$ of $T(d)$ since $T(d) \ge T^{\text{path}}(d)$.
We could find a better bound, by constructing an infinite family of stacked $d$-spheres showing $T(d) \ge \frac{2d+3}{(d+2)^2}$ for $d \ge 3$ and $T(2) \ge \frac{6}{13}$.
Note that this implies that $T(2)$ is strictly larger than $T^{\text{path}}(2)=\frac{3}{7}$.

\begin{theorem}[Lower bound]\label{thm:lowergeneral}
Let $d \ge 2$ and $k \ge 1$ be integers.
\begin{enumerate}
    \item There exists a stacked $d$-sphere $S_2$ on $13k$ vertices such that $\tau(S_2) \ge 6k$.
    \item There exists a stacked $d$-sphere $S_d$ on $(d+2)^2k$ vertices such that $\tau(S_d) \ge (2d+3)k$.
\end{enumerate}
\end{theorem}

We remark that there exists simplicial $(d+1)$-polytopes with transversal ratio $\frac{2}{d+2}$ on arbitrary many vertices~\cite{BDL21+}.
Moreover, even-dimensional cyclic polytopes can have transversal ratio arbitrary close to $\frac{1}{2}$~\cite{BDL21+, HPT}.
Some \emph{simplicial spheres} require more vertices to transverse all their facets;
there exists infinitely many simplicial 3-spheres with transversal ratio $\frac{11}{21}$~\cite{BDL21+} and simplicial $2d$-spheres with transversal ratio $\frac{2}{5}-o(1)$~\cite{NZ22+}.
On the other hand, stacked spheres have relatively small transversal number since they are constructed in a more restrictive way.

The rest of paper is organized as follows. In Section~\ref{sec:lower}, we provide explicit constructions showing Theorem~\ref{thm:lowerpath} and Theorem~\ref{thm:lowergeneral} respectively. We also introduce gluing lemmas that can be used to construct an infinite family of stacked spheres. The proof of Theorem~\ref{thm:upper} is presented in Section~\ref{sec:upper}.

\section{Constructions for lower bounds}\label{sec:lower}
In this section, we construct infinite families of stacked $d$-spheres showing Theorem~\ref{thm:lowerpath} and Theorem~\ref{thm:lowergeneral} respectively.
The constructions are obtained by ``gluing'' a number of copies of certain stacked spheres of small size.
We first present lemmas that describe how to glue two stacked spheres.
A similar idea to the following lemma was used in \cite[Section 5]{BDL21+}.

\begin{lemma}\label{lem:gluing-gen}
Let $S,T$ be stacked $d$-spheres with $|V(S)|=m$ and $|V(T)|=n$.
Suppose there are facets $f \in S$ and $g \in T$ such that $\tau(S-\{f\})=k$ and $\tau(T-\{g\})=\ell$. 
Then there is a stacked $d$-sphere $K$ on $m+n$ vertices and a facet $h \in K$ such that $\tau(K-\{h\})\ge k+\ell$.
\end{lemma}
\begin{proof}
We assume $V(S)$ and $V(T)$ are disjoint.
Let $S=\partial(\alpha_1,\dots,\alpha_{m-d-1})$ and $T=\partial(\beta_1,\dots,\beta_{n-d-1})$ for some $(d+1)$-simplices $\alpha_i,\beta_j$.
We will first show that we may assume that $f \in \alpha_{m-d-1}$ and $g \in\beta_1$ by reordering and relabeling $\alpha_i, \beta_j$'s.
Suppose $g \in \beta_i$ for some $i \in [n-d-1]$.
Let $B$ be the stacked $(d+1)$-ball where $\Gamma(B)=(\beta_1,\dots,\beta_{n-d-1})$.
Consider the dual graph $G(B)$ of $B$.
We can take an ordering $j_1,j_2,\dots,j_{n-d-1}$ of $[n-d-1]$ such that $i=j_1$ and $(\beta_{j_1},\dots,\beta_{j_{n-d-1}})$ defines the same stacked $(d+1)$-ball $B$.
(For instance, take a depth-first search ordering of $G(B)$ starting from $i$.)
Therefore, by reordering and relabeling, we may assume $g \in \beta_1$.
Also assume $f \in \alpha_{m-d-1}$ following the same logic.

Let $f=\{v_1,\dots,v_{d+1}\}$ and $g=\{w_1,\dots,w_{d+1}\}$.
For each $i\in [d+1]$, let $\gamma_i$ be the $(d+1)$-simplex on $\{v_{i},\dots,v_{d+1}\}\cup\{w_1,\dots,w_i\}$.
Let $K$ be the stacked $d$-sphere on $m+n$ vertices defined by $K=\partial(\alpha_1,\dots,\alpha_{m-d-1},\gamma_1,\dots,\gamma_{d+1},\beta_1,\dots,\beta_{n-d-1})$.
Let $h=\gamma_1\setminus\{v_2\}$.
Then $h$ is a facet of $K$.
Since $K-\{h\}$ contains all facets of $S-\{f\}$ and $T-\{g\}$, we have $\tau(K-\{h\})\ge \tau(S-\{f\})+\tau(T-\{g\})=k+\ell$ as desired.
\end{proof}

Using Lemma~\ref{lem:gluing-gen} repeatedly, we can obtain the following corollary.

\begin{corollary}\label{cor:gluing-gen}
Let $S$ be a stacked $d$-sphere with $|V(S)|=n$.
If $\tau(S-\{f\})=m$ for some facet $f$ of $S$, then there is a stacked $d$-sphere with $nk$ vertices and the transversal number at least $mk$ for each integer $k \ge 1$.
\end{corollary}

By using the same construction idea of Lemma~\ref{lem:gluing-gen}, we can also obtain the following corollary for linear stacked spheres.

\begin{corollary}\label{cor:gluing-path}
Let $S$ be a linear stacked $d$-sphere with $|V(S)|=n$.
Suppose $S=\partial(\alpha_1, \dots, \alpha_{n-d-1})$ for some $(d+1)$-simplices $\alpha_1,\dots,\alpha_{n-d-1}$.
If $\tau(S-\{f,g\})=m$ for some facets $f,g$ of $S$ where $f \in \alpha_1$ and $g \in \alpha_{n-d-1}$, then there is a linear stacked $d$-sphere with $nk$ vertices and the transversal number at least $mk$ for each integer $k \ge 1$.
\end{corollary}

\begin{figure}[htbp]
    \centering
    \includegraphics[scale=0.2]{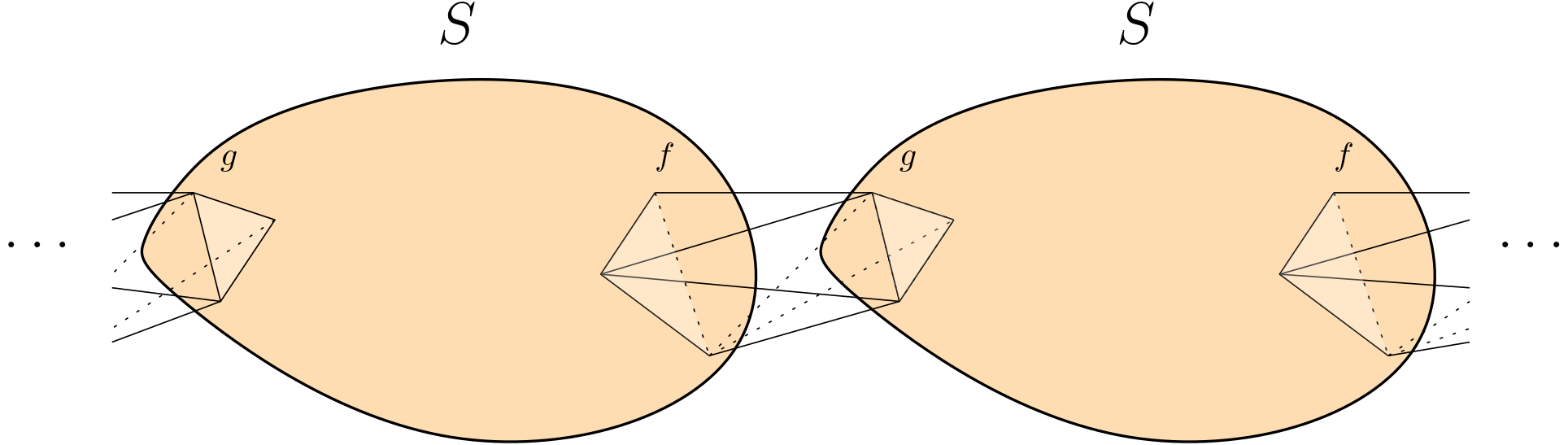}
    \caption{Gluing two copies of a linear stacked sphere $S$ in Corollary~\ref{cor:gluing-path}.}
    \label{fig_2}
\end{figure}

Now we construct an infinite family of linear stacked $d$-spheres showing Theorem~\ref{thm:lowerpath}.

\begin{proof}[Proof of Theorem~\ref{thm:lowerpath}]
For each $i \in [2d+7]$, let $\alpha_i=\{i,i+1,\dots,d+i+1\}$ and $\overline{\alpha}_i$ be the $(d+1)$-simplex on $\alpha_i$.
Let $S=\partial(\overline{\alpha}_1,\dots,\overline{\alpha}_{2d+7})$.
Then $S$ is a linear stacked $d$-sphere on $[3d+8]$.
Let $f=\alpha_1 \setminus \{d+1\}$ and $g=\alpha_{2d+7}\setminus\{2d+8\}$.
We claim that $\tau(S-\{f,g\}) \ge 6$.
Then by Corollary~\ref{cor:gluing-path}, there is a linear stacked $d$-sphere with $(3d+8)k$ vertices and the transversal number at least $6k$ for each integer $k \ge 1$, which completes the proof.

To show $\tau(S-\{f,g\}) \ge 6$ by contradiction, assume $T \subset [3d+8]$ is a transversal of $S-\{f,g\}$ with size $5$.
First, we will show $|T \cap [d+3]| \ge 2$.
Let $\gamma_1=\alpha_2\setminus\{3\}=\{2\} \cup \{4,5,\dots,d+3\}$, $\gamma_2=\gamma_{d+3}=\alpha_1\setminus\{2\}=\{1\}\cup\{3,4,\dots,d+2\}$, and $\gamma_i=\alpha_2\setminus\{i\}$ for $3 \le i \le d+2$.
Then $\gamma_i \subset [d+3]\setminus\{i\}$ is a $d$-face of $S-\{f,g\}$ for each $i \in [d+3]$.
Since $T$ is a transversal of $S-\{f,g\}$, to cover all $\gamma_i$'s, we can conclude that $|T \cap [d+3]| \ge 2$.
Similarly, we can also obtain that $|T \cap \{2d+6,\dots,3d+8\}| \ge 2$.

Then we have $|T \cap \{d+4,\dots,2d+5\}| \le 1$.
Since $\alpha_{d+4}\setminus\{i\}=\{d+4,\dots,2d+5\}\setminus\{i\}$ is a $d$-face of $S-\{f,g\}$ for each $d+5 \le i \le 2d+4$, it must be $|T \cap \{d+4,\dots,2d+5\}|=1$ and $T \cap \{d+4,\dots,2d+5\}$ is either $\{d+4\}$ or $\{2d+5\}$.
Without loss of generality, we may assume $T \cap \{d+4,\dots,2d+5\}=\{2d+5\}$.

Now we have $|T \cap [d+3]|=|T \cap \{2d+6,\dots,3d+8\}|=2$.
Since $\alpha_{d+3}\setminus\{d+4\}=\{d+3\}\cup\{d+5,d+6,\dots,2d+4\}$ and $\alpha_{d+2}\setminus\{d+3\}=\{d+2\}\cup\{d+4,d+5,\dots,2d+3\}$ are $d$-faces of $S-\{f,g\}$ and $T \cap \{d+4,\cdots,2d+4\}=\emptyset$, we obtain that $d+2,d+3 \in T$.
Thus $T\cap [d+3]=\{d+2,d+3\}$ and hence $T \cap [d+1]=\emptyset$.
However, $[d+1]=\alpha_1\setminus\{d+2\}$ is a $d$-face of $S-\{f,g\}$, which contradicts the assumption.
\end{proof}

Next, we give a construction of infinite family of non-linear stacked $d$-spheres proving Theorem~\ref{thm:lowergeneral}.

\begin{proof}[Proof of Theorem~\ref{thm:lowergeneral}]
First, we construct stacked $d$-spheres for each $d \ge 2$ showing (2), and next we give a construction of (1), which can be obtained from the previous construction of (2) with a slight modification.

\textbf{Proof of (2):}
Fix an integer $d \geq 2$. 
By Corollary~\ref{cor:gluing-gen}, to prove (2), it is enough to find a stacked $d$-sphere $S_d$ on $(d+2)^2$ vertices such that $\tau(S_d -\{f\})\ge 2d+3$ for some facet $f$ of $S_d$.
Let $V$ be a set defined as
$$
V := [d+2] \cupdot \{a^{(i)}_j : i \in [d+1] \textrm{ and } j \in [d+2]\}
$$
where $a^{(i)}_j$'s are all distinct elements. 
Note that $|V|=(d+2)^2$.
We construct a stacked $d$-sphere $S_d$ on vertex set $V$. 
For each $i \in [d+1]$, let $\varphi_i \in \pi_{d+2}$ be a permutation on $[d+2]$ that will be chosen later. For each $j \in [d+1]$, define $\tau^{(i)}_j$ as
\[
\tau^{(i)}_j := \{\varphi_i(j+1), \varphi_i(j+2), \ldots, \varphi_i(d+2), a^{(i)}_1, a^{(i)}_2, \ldots, a^{(i)}_j\},
\]
and $\tau^{(i)}_{d+2} := \{a^{(i)}_1, a^{(i)}_2, \ldots, a^{(i)}_{d+2}\}$. 
Let $B$ be the stacked $(d+1)$-ball consists of $(d+1)$-simplices $\overline{[d+2]}$ and all $\overline{\tau}^{(i)}_j$ for $i \in [d+1]$, $j \in [d+2]$. Finally, let $S_d = \partial B$ be the stacked $d$-sphere of $B$.

We claim that there exists an appropriate choice of $\varphi_1, \varphi_2, \ldots, \varphi_{d+1} \in \pi_{d+2}$ and facet $f \in S_d$ such that $\tilde{S}_d := S_d - \{f\}$ has the transversal number at least $2d+3$.

Let $T$ be a transversal of $S_d$. The following observation also applies to transversals of $\tilde{S}_d$ in a similar way.

\begin{obs}\label{o:TransversalatLeaf}
$\size{T \cap \tau^{(i)}_{d+2}} \geq 1$ for each $i \in [d+1]$ and if $\size{T \cap \tau^{(i)}_{d+2}} = 1$ for some $i \in [d+1]$, then we have $\varphi_i(d+1), \varphi_i(d+2) \in T$.
\end{obs}

\begin{proof}[Proof of Observation~\ref{o:TransversalatLeaf}]
Since $\{a^{(i)}_2, a^{(i)}_3, \ldots, a^{(i)}_{d+2}\} \subseteq \tau^{(i)}_{d+2}$ is a facet of $S_d$, we have $\size{T \cap \tau^{(i)}_{d+2}} \geq 1$ for each $i \in [d+1]$.

Suppose $\size{T \cap \tau^{(i)}_{d+2}} = 1$ for some $i \in [d+1]$.
Let $a^{(i)}_j$ be the common element of $T$ and $\tau^{(i)}_{d+2}$. If $j < d+2$, then $a^{(i)}_j$ avoids the facet $\tau^{(i)}_{d+2} - \{a^{(i)}_j\}$ of $S_d$. Hence $j = d+2$. Now consider $(d+1)$-subsets of $\tau^{(i)}_{d+1} = \{\varphi_i(d+2), a^{(i)}_1, a^{(i)}_2, \ldots, a^{(i)}_{d+1}\}$. If $\varphi_i(d+2) \notin T$, then $T$ does not intersect with any $(d+1)$-subset of $\tau^{(i)}_{d+1}$ in $S_d$. Thus, we obtain $\varphi_i(d+2) \in T$.
Finally, $T$ intersects with $\{\varphi_i(d+1), a^{(i)}_1, a^{(i)}_2, \ldots, a^{(i)}_{d}\}$ which is a $(d+1)$-subset of $\tau^{(i)}_d$, hence $\varphi_i(d+1) \in T$.
\renewcommand{\qedsymbol}{$\blacksquare$}
\end{proof}

Now we choose permutations $\varphi_1, \varphi_2, \ldots, \varphi_{d+1} \in \pi_{d+2}$ and a facet $f \in S_d$ that will be removed. Let $\varphi_i(j) = i + j$ mod $d+2$. Two $(d+1)$-simplices $\overline{[d+2]}$ and $\overline{\tau}^{(i)}_1$ are attached along the $d$-face $\{\varphi_i(2), \varphi_i(3)$, $\ldots$, $\varphi_i(d+2)\}$, and this $d$-face is distinct for every $i \in [d+1]$ by our choice of $\varphi_i$'s.
Thus $S_d$ is well-defined.
Let $f$ be any facet of $S_d$ that does not appear in the proof of Observation~\ref{o:TransversalatLeaf}, for instance $f = \{\varphi_1(d+1), \varphi_1(d+2), a^{(1)}_2, \ldots, a^{(1)}_{d}\}$ which is a $(d+1)$-subset of $\tau^{(1)}_d$.

Let $\tilde{S}_{d} := S_d - \{f\}$. Our claim is that the transversal number of $\tilde{S}_d$ is at least $2d+3$. Let $\tilde{T}$ be a transversal of $\tilde{S}_d$. Then the proof of Observation~\ref{o:TransversalatLeaf} also works for $\tilde{S}_d$ and $\tilde{T}$.

We estimate the size of $\tilde{T}$. Let $t=|\tilde{T} \cap [d+2]|$.
Since we attached $(d+1)$ many simplices $\overline{\tau}^{(1)}_1,\dots,\overline{\tau}^{(d+1)}_1$ to $\overline{[d+2]}$, there is a $(d+1)$-subset of $[d+2]$ that is a facet of $S_d$.
Thus we obtain that $t \ge 1$.

Now, there are at most $t-1$ indices $i \in [d+1]$ such that both $\varphi_i(d+1), \varphi_i(d+2) \in \tilde{T}$ by our choice of $\varphi_i$'s. 
Then for those at most $t-1$ indices $i$, we have $\size{\tilde{T} \cap \tau^{(i)}_{d+2}} \ge 1$, and for the other at least $d-t+2$ indices $i$, we have $\size{\tilde{T} \cap \tau^{(i)}_{d+2}} \ge 2$ by Observation~\ref{o:TransversalatLeaf}. 
Since $[d+2] \cupdot \tau^{(1)}_{d+2} \cupdot \cdots \cupdot \tau^{(d+1)}_{d+2}$ forms a partition of $V$, we conclude that 
\begin{align*}
    \size{\tilde{T}} &= \size{\tilde{T}\cap[d+2]}+\sum_{i \in [d+1]}\size{\tilde{T}\cap\tau^{(i)}_{d+2}}\\
    &\geq t+(t-1) \cdot 1 + (d-t+2) \cdot 2 =2d+3.
\end{align*}

\textbf{Proof of (1):}
By Corollary~\ref{cor:gluing-gen}, it suffices to construct a stacked $2$-sphere $S_2$ on $13$ vertices such that $\tau(S_2 -\{f\})\ge 6$ for some facet $f$ of $S_2$.

We construct a stacked $2$-sphere $S_2$ on the following vertex set $V$.
$$
V := \{1,2,3,4\} \cupdot \{a^{(i)}_j : i \in [3] \textrm{ and } j \in [3]\}
$$
where $a^{(i)}_j$'s are all distinct elements. 
Note that $|V|=13$.

For each $i \in [3]$, let $\varphi_i \in \pi_{4}$ be a permutation on $[4]$ that will be chosen later. For each $j \in [3]$, define $\tau^{(i)}_j$ as
\[
\tau^{(i)}_j = \{\varphi_i(j+1), \ldots, \varphi_i(4), a^{(i)}_1, \ldots, a^{(i)}_j\}.
\]
Let $B$ be the stacked 3-ball consists of 3-simplices $\overline{[4]}$ and all $\overline{\tau}^{(i)}_j$ for $i \in [3]$, $j \in [3]$. Finally, let $S_2 = \partial B$ be the stacked $2$-sphere of $B$.

Let $T$ be a transversal of $S_2$. We need the following analogue of Observation~\ref{o:TransversalatLeaf}.

\begin{obs}\label{o:TransversalatLeaf2}
For every $i \in [3]$, $\size{T \cap \{a^{(i)}_1, a^{(i)}_2, a^{(i)}_3\}} \geq 1$. If $\size{T \cap \{a^{(i)}_1, a^{(i)}_2, a^{(i)}_3\}} = 1$ for some $i \in [3]$, then we have either $\varphi_i(4) \in T$ or $\{\varphi_i(2), \varphi_i(3)\}\subseteq T$.
\end{obs}

\begin{proof}[Proof of Observation~\ref{o:TransversalatLeaf2}]
Since $\{a^{(i)}_1, a^{(i)}_2, a^{(i)}_3\} \subseteq \tau^{(i)}_3$ is a facet of $S_2$, it is obvious that $\size{T \cap \{a^{(i)}_1, a^{(i)}_2, a^{(i)}_3\}} \geq 1$ for each $i \in [3]$.

Suppose $\size{T \cap \{a^{(i)}_1, a^{(i)}_2, a^{(i)}_3\}} = 1$ for some $i \in [3]$ and $\varphi_i(4) \notin T$. 
Since $\tau^{(i)}_3\setminus\{a^{(i)}_1\}=\{\varphi_i(4), a^{(i)}_2, a^{(i)}_3\}$ and $\tau^{(i)}_3\setminus\{a^{(i)}_2\}=\{\varphi_i(4), a^{(i)}_1, a^{(i)}_3\}$ are facets of $S_2$, we can conclude that $T \cap \{a^{(i)}_1, a^{(i)}_2, a^{(i)}_3\} = \{a^{(i)}_3\}$. 
Now consider facets $\{\varphi_i(3), a^{(i)}_1, a^{(i)}_2\} \subseteq \tau^{(i)}_2$ and $\{\varphi_i(2), \varphi_i(4), a^{(i)}_1\} \subseteq \tau^{(i)}_1$ of $S_2$. Since $T$ intersects both of them but $T \cap \{\varphi_i(4), a^{(i)}_1, a^{(i)}_2\} = \emptyset$, we conclude $\varphi_i(2) \in T$ and $\varphi_i(3) \in T$.
\renewcommand{\qedsymbol}{$\blacksquare$}
\end{proof}

Now we choose permutations as $\varphi_1 = 1234$, $\varphi_2 = 4123$ and $\varphi_3 = 3421$.
Let $f$ be any facet of $S_d$ that does not appear in the proof of Observation~\ref{o:TransversalatLeaf2}, for instance $f = \{1, 3, 4\}$.

Let $\tilde{S}_{2} := S_2 - \{f\}$. Our claim is that the transversal number of $\tilde{S}_2$ is at least $6$.
Let $\tilde{T}$ be a transversal of $\tilde{S}_2$. Then the proof of Observation~\ref{o:TransversalatLeaf2} also works for $\tilde{S}_2$ and $\tilde{T}$.

We estimate the size of $\tilde{T}$. Let $t=|\tilde{T} \cap [4]|$. From a simple case analysis and by our choice of $\varphi_i$'s, one can show that there are at most $t$ indices $i \in [3]$ satisfying either $\varphi_i(4) \in T$ or $\{\varphi_i(2), \varphi_i(3)\} \subseteq T$.
Then for those indices $i$, we have $\size{\tilde{T} \cap \{a^{(i)}_1, a^{(i)}_2, a^{(i)}_3\}} \ge 1$, and for the other at least $3 - t$ indices $i$, we have $\size{\tilde{T} \cap \{a^{(i)}_1, a^{(i)}_2, a^{(i)}_3\}} \ge 2$ by Observation~\ref{o:TransversalatLeaf2}.

Since $[4] \cupdot \{a^{(1)}_1, a^{(1)}_2, a^{(1)}_3\} \cupdot \{a^{(2)}_1, a^{(2)}_2, a^{(2)}_3\} \cupdot \{a^{(3)}_1, a^{(3)}_2, a^{(3)}_3\}$ forms a partition of $V$, we conclude that
\begin{align*}
    \size{\tilde{T}} &= \size{\tilde{T}\cap[4]}+\sum_{i \in [3]}\size{\tilde{T}\cap\{a^{(i)}_1, a^{(i)}_2, a^{(i)}_3\}}\\
    &\geq t+ t \cdot 1 + (3-t) \cdot 2 =6.
    \qedhere
\end{align*}
\end{proof}

\section{Tight upper bound for stacked $2$-spheres}\label{sec:upper}
In this section, we present the proof of Theorem~\ref{thm:upper} which gives an optimal upper bound for stacked $2$-spheres.

\begingroup
\def\thetheorem{\ref{thm:upper}}
\begin{theorem}
Let $n \ge 4$ be an integer and let $S$ be a linear stacked $2$-sphere on $n$ vertices. Then there exists a transversal $T \subseteq V(S)$ of $S$ using at most $\left\lceil\frac{3}{7}n\right\rceil$ vertices.
\end{theorem}
\addtocounter{theorem}{-1}
\endgroup

We prove Theorem~\ref{thm:upper} by induction on the number of vertices.
To do that, we use the following lemma, which says whenever we add 7 more vertices we can take 3 vertices so that together with them we obtain a transversal.

\begin{lemma}\label{l:S}
Let $S$ be a stacked $2$-sphere on $[10] = \{1, 2, \ldots, 10\}$ and $S=\partial(\overline{\sigma}_1,$ $\dots,$ $\overline{\sigma}_7)$ 
where $\sigma_i=\{\alpha_i,\beta_i,\gamma_i,i+3\}$ for $\alpha_i,\beta_i, \gamma_i < i+3$.
For every subset $L \subseteq \{1, 2, 3\}$ of size 2, 
there are transversals $T_1,T_2$ of $S$ with size at most $4$ such that $T_1\cap L, T_2\cap L \neq \emptyset$ and $|\sigma_7 \cap (T_1\cup T_2)| \ge 3$.
\end{lemma}

We first show Theorem~\ref{thm:upper} using Lemma~\ref{l:S}.
\begin{proof}[Proof of Theorem~\ref{thm:upper}]
Let $S=\partial(\overline{\sigma}_1,\dots,\overline{\sigma}_{n-3})$ for some $3$-simplices $\overline{\sigma}_1,\dots,\overline{\sigma}_{n-3}$.
We prove Theorem~\ref{thm:upper} by showing the following stronger statement:
\\
\textit{There are transversals $T_1,T_2$ of $S$ with size at most $\left\lceil\frac{3}{7}n\right\rceil$ such that $T_1 \cup T_2$ contains at least $3$ vertices of $\sigma_{n-3}$.}

We proceed by induction on $n$.
We first check the case when $4 \le n \le 10$.
We show the following by induction on $n$ for $4 \le n \le 10$: \\
There are transversals $T_1,T_2$ with size at most $\left\lfloor\frac{n}{2}\right\rfloor$ such that $T_1 \cup T_2$ contains at least $3$ vertices of $\sigma_{n-3}$.

It is obvious when $n=4$ since any vertex subset of size $2$ is a transversal.
When $n=5$, let $\sigma_1=\{a_1,a_2,a_3,a_4\}$ and $\sigma_2=\{a_2,a_3,a_4,a_5\}$ for some vertices $a_1,\dots,a_5$.
Then $T_1=\{a_2,a_3\}$ and $T_2=\{a_3,a_4\}$ are transversals of $S$ that we wanted.

Suppose $6 \le n \le 10$.
Let $\tilde{S}=\partial(\overline{\sigma}_1,\dots,\overline{\sigma}_{n-5})$.
Then $\tilde{S}$ is a stacked $2$-sphere on $n-2$ vertices.
Let $\sigma_{n-4}=\{v_1,v_2,v_3,v_4\}$ and $\sigma_{n-3}=\{v_2,v_3,v_4,v_5\}$ for some vertices $v_1,\dots,v_5$.
We can assume that $\overline\sigma_{n-4}$ is attached to $\overline\sigma_{n-5}$ along $\{v_1,v_2,v_3\}$.
By the induction hypothesis, there are transversals $T_1,T_2$ of $\tilde{S}$ with size at most $\left\lfloor\frac{n}{2}\right\rfloor-1$ such that $T_1\cup T_2$ contains at least $2$ elements of $\{v_1,v_2,v_3\}$.

\emph{Case 1.} $v_1 \in T_1 \cup T_2$. 
By symmetry, we may assume $\{v_1, v_2\} \subseteq T_1 \cup T_2$ and $v_1 \in T_1$ and $v_2 \in T_2$.
Then $T_1\cup\{v_5\}$ and $T_2\cup \{v_4\}$ are transversals of $S$ and $T_1\cup T_2\cup\{v_4,v_5\}$ contains at least $3$ vertices of $\sigma_{n-3}$.
\smallskip

\emph{Case 2.} $v_1 \notin T_1 \cup T_2$.
It implies $\{v_2, v_3\} \subseteq T_1 \cup T_2$. We can assume $v_2 \in T_1$ and $v_3 \in T_2$.
Then $T_1 \cup \{v_4\}$ and $T_2 \cup \{v_4\}$ are transversals of $S$ and $T_1\cup T_2\cup\{v_4\}$ contains at least $3$ vertices of $\sigma_{n-3}$.
\smallskip

Thus, in both cases, we found two transversals of $S$ with size at most $\left\lfloor\frac{n}{2}\right\rfloor$.
Note that $\left\lfloor\frac{n}{2}\right\rfloor \le \left\lceil\frac{3}{7}n\right\rceil$ for $4 \le n \le 10$.
Therefore, when $4 \le n \le 10$, there are transversals $T_1,T_2$ of $S$ with size at most $\left\lceil\frac{3}{7}n\right\rceil$ such that $T_1\cup T_2$ contains at least $3$ vertices of $\sigma_{n-3}$.

Now, suppose $n >10$.
Let $S_0=\partial(\overline{\sigma}_1,\dots,\overline{\sigma}_{n-10})$.
Then $S_0$ is a stacked $2$-sphere on $n-7$ vertices.
Thus, by the induction hypothesis, we can take a transversal $T_1,T_2$ of $S_0$ with size at most $\left\lceil\frac{3}{7}n\right\rceil-3$ such that $T_1\cup T_2$ contains at least $3$ vertices of $\sigma_{n-10}$.

Now let $S_1=\partial(\overline{\sigma}_{n-9},\dots,\overline{\sigma}_{n-3})$.
Then $S_1$ is a stacked $2$-sphere on $10$ vertices, and $S$ is obtained by gluing $S_0$ and $S_1$ along the face $\sigma_{n-9}\cap\sigma_{n-10}$.
Since $T_1\cup T_2$ contains at least $3$ vertices of $\sigma_{n-10}$, $T_1\cup T_2$ contains at least $2$ vertices of $\sigma_{n-9}\cap\sigma_{n-10}$.
Let $L$ be a subset of $(T_1\cup T_2) \cap \sigma_{n-9}\cap\sigma_{n-10}$ with size $2$.
Then by Lemma~\ref{l:S}, there are transversals $W_1,W_2$ of $S_1$ with at most $4$ vertices such that $W_1 \cap L,W_2\cap L \neq \emptyset$ and $|\sigma_{n-3} \cap (W_1\cup W_2)| \ge 3$.
Since $W_1 \cap L,W_2\cap L \neq \emptyset$, we know that $W_1 \cap T_{i_1} \neq \emptyset$ and $W_2 \cap T_{i_2} \neq \emptyset$ for some $i_1,i_2 \in \{1,2\}$.
Then $W_1 \cup T_{i_1}$ and $W_2 \cup T_{i_2}$ are transversals of $S$ with size at most $\left\lceil\frac{3}{7}n\right\rceil$ and $W_1 \cup W_2 \cup T_{i_1} \cup T_{i_2}$ contains at least $3$ vertices of $\sigma_{n-3}$.
This completes the proof.
\end{proof}

Now it remains to prove Lemma~\ref{l:S}.
\begin{proof}[Proof of Lemma~\ref{l:S}]

    Consider the middle 3-simplex $\overline{\sigma}_4$, and let $\sigma_4=\{w_1, w_2, w_3,$ $w_4\}$. Two 3-simplices $\overline{\sigma}_3$ and $\overline{\sigma}_5$ are glued to $\overline{\sigma}_4$, and assume that $\sigma_3 \cap \sigma_4 =$ $\{w_1,$ $w_2,$ $w_3\}$ and $\sigma_4 \cap \sigma_5 = \{w_2, w_3, w_4\}$. Note that the edge $e_0 = \{w_2, w_3\}$ transverses all facets of the stacked $2$-sphere $\partial(\overline{\sigma}_3, \overline{\sigma}_4,\overline{\sigma}_5)$. We try to find as many transversals of the form $T = e_0 \cup \{u, v\}$ as possible,
    where $u \in \sigma_1 \cup \sigma_2$ and $v\in \sigma_6 \cup \sigma_7$. If we cannot find two promised $T_1, T_2$ of this form, then we analyze the structure of $S$ and find them in alternative ways.

    Denote $\sigma_4 = \{w_1, w_2, w_3, w_4\}$ as above, and let $\sigma_1 = \{u_1, u_2, u_3, u_4\}$, $\sigma_2 = \{u_2, u_3,$ $ u_4, u_5\}$, $\sigma_7 = \{v_1, v_2, v_3, v_4\}$ and $\sigma_6 = \{v_2, v_3, v_4, v_5\}$ (see Figure~\ref{fig_3}). Note that some $u, v$ and $w$'s can be repeated. Without loss of generality, we may assume $\{u_1, u_2, u_3\} = \{1, 2, 3\}$ so that $L \subseteq \{u_1, u_2, u_3\}$.
\begin{figure}[htbp]
    \centering
    \includegraphics[scale=0.3]{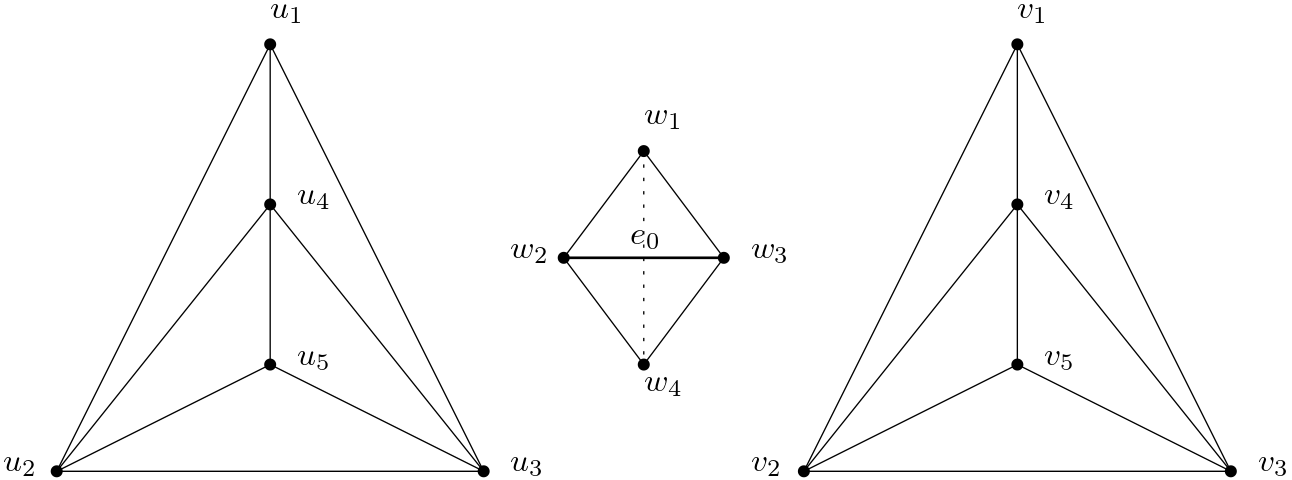}
    \caption{$\overline{\sigma}_1 \cup \overline{\sigma}_2$(left), $\overline{\sigma}_4$(middle) and $\overline{\sigma}_6 \cup \overline{\sigma}_7$(right)}
    \label{fig_3}
\end{figure}

    By the definition of $e_0 = \{w_2, w_3\}$, we have $e_0 \subseteq \sigma_3$. Since $\size{\sigma_3 \setminus \sigma_2} = 1$, $\sigma_2$ contains an element of $e_0$.
    \begin{obs}\label{o:sigma1and2}
        Unless $e_0 \cap \sigma_2 = \{u_5\}$ and $L = \{u_2, u_3\}$, we can select some $u \in \sigma_1 \cup \sigma_2$ so that $e_0 \cup \{u\}$ intersects with all facets of the stacked $2$-sphere $\partial(\overline{\sigma}_1, \overline{\sigma}_2)$ and $(e_0 \cup \{u\}) \cap L \neq \emptyset$.
    \end{obs}
    \begin{proof}[Proof of Observation~\ref{o:sigma1and2}]
        Recall that $\sigma_2 = \{u_2, u_3,$ $ u_4, u_5\}$. We divide cases up to the vertices in $e_0 \cap \sigma_2$.
        \smallskip

        \emph{Case 1.} $u_2 \in e_0 \cap \sigma_2$. Since $\{u_2, u_3\}$ intersects with all facets of $\partial(\overline{\sigma}_1, \overline{\sigma}_2)$ and intersects with $L$ (by the pigeonhole principle), $u = u_3$ is the desired choice.
        \smallskip

        \emph{Case 2.} $u_3 \in e_0 \cap \sigma_2$. As in Case 1, select $u = u_2$.
        \smallskip

        \emph{Case 3.} $u_4 \in e_0 \cap \sigma_2$. Recall that $L \cap \{u_2, u_3\} \neq \emptyset$. Since both $\{u_2, u_4\}$ and $\{u_3, u_4\}$ intersect with every facet of $\partial(\overline{\sigma}_1, \overline{\sigma}_2)$, $u$ can be chosen from any element in $L \cap \{u_2, u_3\}$.
        \smallskip

        \emph{Case 4.} $u_5 \in e_0 \cap \sigma_2$ and $L \neq \{u_2, u_3\}$. The condition on $L$ implies $u_1 \in L$. Choose $u = u_1$.
        \renewcommand{\qedsymbol}{$\blacksquare$}
    \end{proof}

    Except for the case that $e_0 \cap \sigma_2 = \{u_5\}$ and $L = \{u_2, u_3\}$, we have chosen $e_0 \cup \{u\}$ which intersects with every facet of $\partial(\overline{\sigma}_1,\dots,\overline{\sigma}_5)$. Now we turn our attention to choose $v$ in several ways from $\sigma_6 \cup \sigma_7$, yielding the desired transversals $T_1$ and $T_2$. Except for one case, this is done in the following observation which is very similar to Observation~\ref{o:sigma1and2}. First, note that $\sigma_6$ contains an element of $e_0$ by the same argument before Observation~\ref{o:sigma1and2}.

    \begin{obs}\label{o:sigma6and7}
        Unless $e_0 \cap \sigma_6 = \{v_5\}$, there are two sets $W_1,W_2$ of the form $e_0 \cup \{v\}$ where $v \in \sigma_6\cup\sigma_7$ such that each of $W_1,W_2$ intersects with all facets of $\partial(\overline{\sigma}_6,\overline{\sigma_7})$ and $W_1 \cup W_2$ contains at least three vertices of $\sigma_7$.
    \end{obs}
    \begin{proof}
        Again, we divide cases up to the vertices in the intersection $e_0 \cap \sigma_6$.
        \smallskip

        \emph{Case 1.} $v_2 \in e_0 \cap \sigma_6$. $e_0 \cup \{v_3\}$ and $e_0 \cup \{v_4\}$ are the two desired sets.
        \smallskip

        \emph{Case 2.} $v_3 \in e_0 \cap \sigma_6$. $e_0 \cup \{v_2\}$ and $e_0 \cup \{v_4\}$ are the two desired sets.
        \smallskip

        \emph{Case 3.} $v_4 \in e_0 \cap \sigma_6$. $e_0 \cup \{v_2\}$ and $e_0 \cup \{v_3\}$ are the two desired sets.
        \renewcommand{\qedsymbol}{$\blacksquare$}
    \end{proof}

    So far we have shown the lemma except two cases: first when $e_0 \cap \sigma_2 = \{u_5\}$ and $L = \{u_2, u_3\}$, second when $e_0 \cap \sigma_6 = \{v_5\}$. The following observation turns out to be useful when we handle the remaining cases.

    \begin{obs}\label{o:sigma1234}
        If $e_0 \cap \sigma_2 = \{u_5\}$, then $(\sigma_2 \cap \sigma_3) \setminus \{u_5\}$ has size two and transverses all facets of $\partial(\overline{\sigma}_1,\dots,\overline{\sigma}_4)$ but $\{w_2, w_3, w_4\}$. Moreover, it intersects with $L = \{u_2, u_3\}$.
    \end{obs}
    Note $\{w_2, w_3, w_4\}$ is not a facet of $S$ since $\overline{\sigma}_5$ is attached to $\overline{\sigma}_4$ along that face.
    \begin{proof}[Proof of Observation~\ref{o:sigma1234}]
        Recall that $w_1 \in \sigma_3$, $e_0 \subseteq \sigma_3$ and $\size{\sigma_3 \setminus \sigma_2} = 1$. Since $e_0$ $\cap$ $\sigma_2 $ $= \{u_5\}$, we conclude that $w_1 \in \sigma_2$ and $(\sigma_2 \cap \sigma_3) \setminus \{u_5\}$ intersects with all $2$-subsets of $\sigma_4$ but $\{w_2, w_3, w_4\}$.
        Since $(\sigma_2 \cap \sigma_3)\setminus \{u_5\}$ is either $\{u_2, u_3\}$, $\{u_2, u_4\}$ or $\{u_3, u_4\}$, each of them intersects with all facets of $\partial(\overline{\sigma}_1, \overline{\sigma}_2, \overline{\sigma}_3)$ and $L = \{u_2, u_3\}$.
        \renewcommand{\qedsymbol}{$\blacksquare$}
    \end{proof}

    Hence when $e_0 \cap \sigma_2 = \{u_5\}$ and $L = \{u_2, u_3\}$, it lefts to choose two vertices that transverse all facets of $\partial(\overline{\sigma}_5, \overline{\sigma}_6, \overline{\sigma}_7)$ but $\{w_2, w_3, w_4\}$. This is covered in the following observation.

    \begin{obs}\label{o:sigma567}
        There are two subsets $S_1, S_2 \subseteq \sigma_6 \cup \sigma_7$ of size two such that $\size{\sigma_7 \cap (S_1 \cup S_2)} \geq 3$ and each $S_i$ transverses all facets of $\partial(\overline{\sigma}_5, \overline{\sigma}_6, \overline{\sigma}_7)$ but $\{w_2, w_3$, $w_4\}$.
    \end{obs}
    \begin{proof}[Proof of Observation~\ref{o:sigma567}]
        Let $\sigma_5 = \{v_2, v_3, v_5, v_6\}$. (See Figure~\ref{fig_4}.) Because $\sigma_6$ is attached to $\sigma_5$ along $\{v_2, v_3, v_5\}$, the facet $\{w_2, w_3, w_4\}$ which may not be covered by $S_1, S_2$ is either $\{v_2, v_3, v_6\}$, $\{v_2, v_5, v_6\}$ or $\{v_3, v_5, v_6\}$. In each case, we choose $S_1$ and $S_2$ satisfying the observation as follows.
        \smallskip

        \emph{Case 1.} $\{w_2, w_3, w_4\} = \{v_2, v_3, v_6\}$. Take $S_1 = \{v_1, v_5\}$ and $S_2 = \{v_2, v_3\}$.
        \smallskip

        \emph{Case 2.} $\{w_2, w_3, w_4\} = \{v_2, v_5, v_6\}$. Take $S_1 = \{v_2, v_3\}$ and $S_2 = \{v_3, v_4\}$.
        \smallskip

        \emph{Case 3.} $\{w_2, w_3, w_4\} = \{v_3, v_5, v_6\}$. This is symmetric to Case 2, and we take $S_1 = \{v_2, v_3\}$ and $S_2 = \{v_2, v_4\}$.
        \renewcommand{\qedsymbol}{$\blacksquare$}
    \end{proof}
\begin{figure}[htbp]
    \centering
    \includegraphics[scale=0.25]{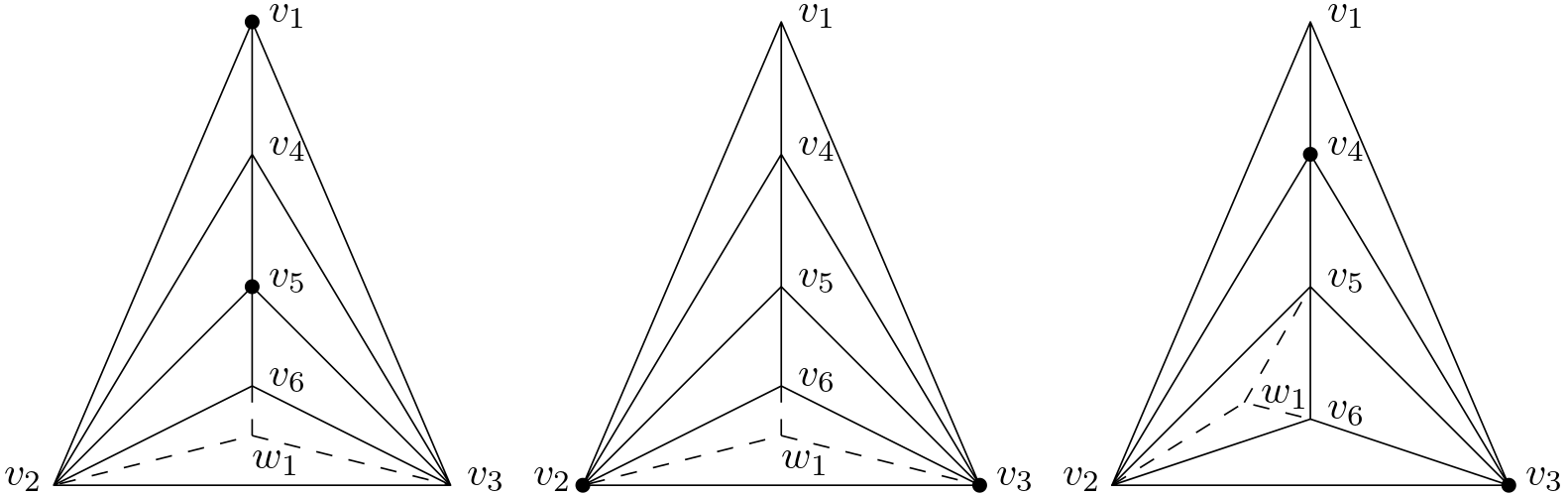}
    \captionsetup{width=.85\linewidth}
    \caption{$\overline{\sigma}_5 \cup \overline{\sigma}_6 \cup \overline{\sigma}_7$ with $\overline{\sigma}_4$ attached along $\{v_2, v_3, v_6\}$ (left, middle) and $\{v_2, v_5, v_6\}$ (right). The marked vertices cover all facets but $\overline{\sigma}_4 \cap \overline{\sigma}_5$.}
    \label{fig_4}
\end{figure}
    Therefore when $e_0 \cap \sigma_2 = \{u_5\}$, $T_i = ((\sigma_2 \cap \sigma_3) \setminus \{u_5\}) \cup S_i$, $i = 1, 2,$ are the desired transversals.

    Finally, we construct $T_1$ and $T_2$ when $e_0 \cap \sigma_6 = \{v_5\}$. Since we already handled the case when $e_0 \cap \sigma_2 = \{u_5\}$ and $L = \{u_2, u_3\}$, we may assume it is not.
    We take the first transversal $T_1 = e_0 \cup \{u, v_1\}$ where $u$ is chosen using Observation~\ref{o:sigma1and2}.
    To construct $T_2$, we apply the same argument of Observation~\ref{o:sigma1234} to $\partial(\overline{\sigma}_4,\dots,\overline{\sigma}_7)$ and $v_i$'s instead, and conclude that $(\sigma_5 \cap \sigma_6) \setminus \{v_5\}$ transverses all facets of $\partial(\overline{\sigma}_4, \dots, \overline{\sigma}_7)$ but $\{w_1, w_2, w_3\}$. Then we use Observation~\ref{o:sigma567} to $\partial(\overline{\sigma}_1, \overline{\sigma}_2, \overline{\sigma}_3)$ to get two subsets $S_1, S_2 \subseteq \sigma_1 \cup \sigma_2$ of size two such that $\size{\sigma_1 \cap (S_1 \cup S_2)} \geq 3$ and each $S_i$ transverses all facets of $\partial(\overline{\sigma}_1, \overline{\sigma}_2, \overline{\sigma}_3)$ but $\{w_1, w_2, w_3\}$. Since $\size{\sigma_1 \cap (S_1 \cup S_2)} \geq 3$ implies that some $S_i$ intersects with $L \subseteq \sigma_1$ which has size two, then $T_2 = ((\sigma_5 \cap \sigma_6) \setminus \{v_5\}) \cup S_i$ is another transversal that we wanted.
\end{proof}

\section*{Acknowledgement}
The authors would like to thank Seunghun Lee for introducing this problem.
The authors also thank Andreas Holmsen and Minki Kim for helpful discussions in the early stage of this work.

\end{document}